\documentclass{article}
\usepackage[utf8]{inputenc}
\usepackage{enumerate}
\usepackage{amssymb, bm}
\usepackage{amsthm}
\usepackage{amsmath}
\usepackage{tikz-cd} 
\usepackage{bbold}
 \usepackage[all]{xy}
\title{On the Nakano vanishing theorem}
\author{Xiaojun WU}
\date{\today}
\newtheorem{mythm}{Theorem}
\newtheorem{mylem}{Lemma}

\newtheorem{mydef}{Definition}
\newtheorem{myrem}{Remark}
\begin{document}
\def\cI{\mathcal{I}}
\def\Z{\mathbb{Z}}
\def\Q{\mathbb{Q}}  \def\C{\mathbb{C}}
 \def\R{\mathbb{R}}
 \def\N{\mathbb{N}}
 \def\H{\mathbb{H}}
  \def\P{\mathbb{P}}
 \def\rC{\mathrm{C}}
  \def\Bs{\mathrm{Bs}}
  \def\d{\partial}
 \def\dbar{{\overline{\partial}}}
\def\dzbar{{\overline{dz}}}
 \def\ii{\mathrm{i}}
  \def\d{\partial}
 \def\dbar{{\overline{\partial}}}
\def\dzbar{{\overline{dz}}}
\def \ddbar {\partial \overline{\partial}}
\def\cK{\mathcal{K}}
\def\cE{\mathcal{E}}  \def\cO{\mathcal{O}}
\def\P{\mathbb{P}}
\def\cI{\mathcal{I}}
\def \loc{\mathrm{loc}}
\def \log{\mathrm{log}}
\def \cC{\mathcal{C}}
\bibliographystyle{plain}
\def \dim{\mathrm{dim}}
\def \RHS{\mathrm{RHS}}
\def \liminf{\mathrm{liminf}}
\def \ker{\mathrm{Ker}}
\def \Null{\mathrm{Null}}
\maketitle
\begin{abstract}
In this note, we state various generalisations of the Nakano vanishing theorem under weak positivity assumptions, and compare them with the known results.
\end{abstract}

\section{Introduction}
In this note, we give the following generalized version of the Nakano vanishing theorem.
\begin{mythm}
Let $X$ be a $n$-dimensional projective manifold and $L$ a nef holomorphic line bundle over $X$. Then we have
$$H^p(X, \Omega^q_X \otimes L)=0$$
for any $p+q > n+\max(\dim (B_+(L)),0)$.
Here $B_+(L)$ denotes the augmented base locus(or non-ample locus) of $L$. When $B_+(L)=\emptyset$, we define by convention that its dimension is $-1$.
\end{mythm}
Here we recall the definition of $B_+(L)$. Given an ample line bundle $A$ over $X$.
the augmented base locus is defined by
$$B_+(L):= \bigcap_{m>0}\mathrm{Bs} (mL-A)$$
where $\Bs$ means the base locus of a line bundle.

We recall classically (cf. \cite{BBP}) that $B_+(L) = \emptyset$ if and only if $L$ is ample and $B_+(L) \neq X$ if and only if $L$ is big.
Thus we have the Nakano vanishing theorem in the case that $B_+(L) = \emptyset$.

We notice that by the example of \cite{Ram}, we can not change the augmented base locus by the base locus. In his example, we take $X$ the blow up of $\P^3$ at one point and $L$ the pull back of $\cO_{\P^3}(1)$ under the blow up. 
Thus $L$ is a big and nef line bundle with $\bigcap_{m>0} \Bs(mL) = \emptyset$.
But by calculation of cohomology class we can show that
$$H^2(X, \Omega^2_X \otimes L) \neq 0.$$
We observe that in this example $B_+(L)=E$ where $E$ is the exceptional divisor.

Now, we return to the proof of the theorem. We argue by induction on the dimension of $B_+(L)$ and apply of the Nakamaye theorem.
First note that we can assume $L$ big,
otherwise $B_+(L)=X$ and the theorem is void.

Let $l:= \dim (B_+(L))$. When $l = -1$, the theorem is true by the Nakano vanishing theorem. 
When $l\leq 0$, we show that in fact $L$ is ample. 
In this case, there exists some $m >0$ and $s_0, \cdots, s_k \in H^0(X, mL-A)$ such that 
$$\Bs(s_0, \cdots, s_k)=\{x_0, \cdots, x_l\}.$$
These sections induce a singular metric $h_0$ on $mL-A$ with analytic singularity at the discrete points $\{x_0, \cdots, x_l\}$. Its curvature is a closed positive (1,1)-current which is smooth outside $\{x_0, \cdots, x_l\}$. 
By \cite{Dem92} Lemma 6.3 $mL-A$ is nef. Hence $L$ is ample.

Now let $l>0$ and suppose by induction that the theorem has be verified for $\dim (B_+(L)) \leq l-1$. We recall the concepts involved in the theorem of Nakamaye on base loci \cite{Nak}.
\begin{mydef}
Given a nef and big divisor $L$ on $X$, the null locus $\Null(L)$ of $L$ is the union of all positive dimensional subvarieties $V \subset X$ with
$$(L^{\dim V} \cdot V)=0.$$
\end{mydef}
We observe that for any smooth divisor $D$ of $X$ and such a line bundle,
$$\Null(L|_D) \subset \Null(L).$$
\begin{mythm} {\rm (Nakamaye)}.
If $L$ is an arbitrary nef and big divisor on $X$, then
$$B_+(L)=\Null(L).$$
\end{mythm}

Fix $A_2$ a very ample divisor on $X$. By Bertini theorem with a general choice we can assume that $D \in |A_2|$ is smooth.
Since $A_2$ is very ample we can assume that $D \cap B_+(L) \subsetneq B_+(L)$.
More precisely, for a general choice of $D$, no $l-$dimensional component of $ B_+(L)$ is contained in $D$.
Since $L$ is nef and big, we have by Nakayame theorem $\Null(L)=B_+(L)$.
By the definition of $\Null(L)$ we have
$$(L^{n-1} \cdot D)>0.$$
In other words, $L|_D$ is big.
Using another time the Nakamaye theorem, we find that
$$B_+(L|_D)=\Null(L|_D) \subset \Null(L) \cap D \subsetneq B_+(L).$$
In particular, $\dim B_+(L|_D) \leq \dim B_+(L)-1$.

Recall the following elementary lemma (3.24) in \cite{SS}.
\begin{mylem}
Let $L$ be a holomorphic line bundle over $X$, let $D$ be a smooth hyper-surface in $X$, and let $p,q \geq 0$ be fixed. If
$$(a) H^p(X, \Omega_X^q \otimes [D] \otimes L)=0,$$
$$(b) H^{p-1}(D, \Omega_D^{q-1} \otimes L|_D)=0,$$
$$(c) H^{p-1}(D, \Omega_D^{q} \otimes ([D]\otimes L)|_D)=0,$$
then we have
$$ H^{p}(X, \Omega_X^{q} \otimes L)=0.$$
\end{mylem}
Since $[D] \otimes L$ is ample ($L$ is nef), the hypotheses (a) (c) of the lemma is verified by the Nakano vanishing theorem.
Since
$$(p-1)+(q-1) > \dim D+l-1,$$
the condition (b) is satisfied by the inductive hypothesis.

This finishes the proof.
\begin{myrem}
{\rm
It would be interesting to know whether the theorem is still valid without assuming $L$ to be nef. 
Here principally, we use the nef condition in two places: in the Nakamaye theorem and in the fact that the sum of an ample divisor and a nef divisor is ample.
}
\end{myrem}

Here, following some ideas of Demailly, we give the following more general version of the Nakano vanishing theorem.
\begin{mythm}
Let $X$ be a $n$-dimensional projective manifold, $L$ a holomorphic line bundle and $A$ an ample line bundle over $X$. Assume that for sufficiently large $m \in \N$ and general hyper-surfaces in the linear system $H_1, \cdots, H_k \in |mA|$, the restriction $L|_{H_1 \cap \cdots \cap H_k}$ is ample.
Then for $p+q > n$, we have
$$H^q(X, \Omega^p_X \otimes L)=0.$$
\end{mythm} 
\begin{proof}
By duality, it is equivalent to show that for $p+q < n-k$,we have
$$H^q(X, \Omega^p_X \otimes L^{-1})=0.$$
Since the hyper-surface $H_i$ is supposed to be general, we can assume that any intersection of type $H_1 \cap \cdots \cap H_l$ is smooth for any $l$ and of dimension $n-l$ for any $l \leq k$.

For $m$ big enough such that $mA+L$ is ample, hence by Nakano vanishing theorem we have the vanishing $p+q < n-k$
$$H^q(X, \Omega^p_X \otimes L^{-1} \otimes \cO(-H_1))=0.  $$
From the short exact sequence
$$0 \to \Omega^p_X \otimes L^{-1} \otimes \cO(-H_1) \to \Omega^p_X \otimes L^{-1} \to (\Omega^p_X \otimes L^{-1})|_{H_1} \to 0$$
we know that to prove the desired vanishing it is enough to show that for $p+q < n-k$
$$H^q(X, (\Omega^p_X \otimes L^{-1})|_{H_1})=0.$$
From the short exact sequence
$$0 \to T_{H_1} \to T_X|_{H_1} \to \cO(H_1)|_{H_1} \to 0$$
we have the exact sequence (using the fact that $\cO(H_1)$ is of rank one)
$$0 \to \cO(-H_1)|_{H_1} \otimes \Omega^{p-1}_{H_1} \to \Omega^p_X|_{H_1} \to \Omega^p_{H_1} \to 0.$$
We take the tensor product with $L^{-1}|_{H_1}$ and the long exact sequence associated to the coreesponding short exact sequence.
By the Nakano vanishing theorem, we find
$$H^i(H_1, \Omega^j_{H_1} \otimes (L^{-1} \otimes \cO(-H_1))|_{H_1})=0$$
for any $i+j <n-1$.
It is enough to prove that
$$H^q(H_1, (\Omega^p_{H_1} \otimes L^{-1}|_{H_1})=0$$
for $p+q <n-k$.

We continue this process and change $X$ with $H_1$, then $H_1$ with $H_1 \cap H_2$ etc. Taking from the beginning $m$ so big that $mA+L$ is ample, we get for every $l$ that $mA+L|_{H_1 \cap \cdots \cap H_l}$ is ample on $H_1 \cap \cdots \cap H_l$. 
Hence in each step, we can use the Nakano vanishing theorem.
Finally, we are reduced to proving that
$$H^q(H_1 \cap \cdots \cap H_k, \Omega^p_{H_1 \cap \cdots \cap H_k} \otimes L^{-1}|_{H_1 \cap \cdots \cap H_k})=0$$
for $p+q <n-k$.
But this is true by the Nakano vanishing theorem and our assumption.
\end{proof}
\begin{myrem}
{\rm
By the proof of the theorem, it is enough to take $m$ so large that $mA+L$ is ample, and $H_i \in |mA|$ so that $H_1 \cap \cdots \cap H_l$ is smooth and of dimension $n-l$ for any $l \leq k$, and $L|_{H_1 \cap \cdots \cap H_k}$ is ample.
}
\end{myrem}
As pointed out by A. H\"oring, it is interesting to compare this result to the following theorem 2 of \cite{Kur13}:

Let $X$ be a smooth projective variety, $L$ a divisor, $A$ a very ample
divisor on $X$. If $L|_{E_1 \cap \cdots \cap E_k}$ is big and nef for a general choice of $E_1, \cdots , E_k$, then
$H^i (X,\cO_X(K_X + L)) = 0$ for $i > k$.
\begin{myrem}
{\rm
Our first theorem is a special case of this general version.
Since $L$ is nef, it is nef on the complete intersection of the hyper-surfaces $H_1, \cdots, H_l$ where $l := \dim(B_+(L))$.
On the other hand, for such general hyper-surfaces, we can assume that the intersection $B_+(L) \cap H_1 \cap \cdots \cap H_l$ is finite points.
By the definition of stable base locus, $L|_{H_1 \cap \cdots \cap H_l}$ is ample outside these finite points.
Hence in fact,  $L|_{H_1 \cap \cdots \cap H_l}$ is ample.}
\end{myrem}
The $k$-ampleness condition defined by Sommese \cite{Som} is also a sufficient condition for the condition stated in Theorem 3.30. 
We start by recalling the definition.
\begin{mydef}
A holomorphic line bundle $L$ on a compact complex manifold $X$ is said to be $k$-ample $(0 \leq k \leq n-1)$ if there exists a positive integer $N$ such that $NL$ spans at each point of $X$ and the Kodaira morphism associated to $NL$ has at most $k$-dimensional fibres.
\end{mydef}
Changing $N$ in the definition by a possible large multiple of $N$ we can assume that the Kodaira morphism associated to $NL$ is the Iitaka fibration.
Denote $\Phi: X \to Z$ the fibration where $Z$ is a projective variety. 
Denote $A_{z,j}$ ($z \in Z, j \in \N$) the irreducible components of the fibre of $z$ (i.e. $\Phi^{-1}(z)$).
By a general choice of $H_1$, we can assume that for any $z,j$ the hyper-surface $H_1$ intersecting $A_{z,j}$ defines a divisor of $A_{z,j}$ by the lemma stated below.
Similarly, with a general choice of $H_1, \cdots, H_k$ we can assume that for any $z,j$ $H_1 \cap \cdots \cap H_k \cap A_{z,j}$ is a finite set, by the assumption that $\dim A_{z,j} \leq k$.
In other words, the restriction of the Kodaira morphism
$$\Phi: H_1 \cap \cdots \cap H_k \to Z$$
is a finite morphism. 
Since $L|_{H_1 \cap \cdots \cap H_k}$ is pull back of $\cO(1)$ via $\Phi$, $L|_{H_1 \cap \cdots \cap H_k}$ is ample on ${H_1 \cap \cdots \cap H_k}$.
(Recall that the pull back of an ample line bundle under a finite morphism is ample.)
\begin{mylem}
Let $\Phi: X \to Z$ be the fibration such that all the fibers have dimension $\le k$.
Assume $X$ is projective.
Then there exists $H \subset X$ a general very ample divisor such that the restriction $\Phi_H: H \to Z$ of $\Phi$ on $H$ has all fibers of dimension $\le (k-1)$.
\end{mylem}
\begin{proof}
Denote $A_{z,j}$ ($z \in Z, j \in \N$) the irreducible components of the fibre of $z$ (i.e. $\Phi^{-1}(z)$).
It is equivalent to demand the restriction to each $A_{z,j}$ of the defining section $\sigma$ of $H$ is non trivial.
Let $A$ be an ample divisor on $X$.
Denote $V_{z,j}$ the linear subspace of $H^0(X, mA)$ such that $\sigma|_{A_{z,j}} \equiv 0$.
We want to choose $\sigma$ such that $\sigma \in H^0(X, mA) \smallsetminus \bigcup_{z,j} V_{z,j}$. 
Notice that the family $A_{z,j}$ parametrized by $z,j$ forms a bounded family in the Hilbert scheme of $X$. A sufficient condition to find $\sigma$ as above is that for $m$ large enough
$$\dim Z +\dim V_{z,j} < h^0(X, mA).$$
Without loss of generality, we can assume that $A$ is very ample on $X$. 
Hence, by boundedness, we have for $m$ large enough independent of $z,j$ a
surjective restriction morphism
$$H^0(X, mA) \to H^0(A_{z,j}, mA).$$
As $V_{z,j}$ is the kernel of this morphism, it is enough to take $m$
so large that
$$\dim Z < h^0(A_{z,j}, mA).$$
For $A_{z,j}$ with positive dimension, the regular part of $A_{z,j}$ is a smooth submanifold of $X$. 
Since $A$ is very ample, it generates $1$-jets of the regular part of $A_{z,j}$ at any point.
Hence $H^0(A_{z,j}, NA)$ generates any $m$-fold symmetric product of $1$-jets of $A_{z,j}$ at some regular point.
In other words,
$$h^0(A_{z,j}, mA) > {{m}\choose{\dim A_{z,j}}} \geq m.$$
\end{proof}
\textbf{Acknowledgement} I thank Jean-Pierre Demailly, my PhD supervisor, for his guidance, patience and generosity. 
I would like to thank Andreas Höring for some very useful suggestions on the previous draft of this work.
I would also like to express my gratitude to colleagues of Institut Fourier for all the interesting discussions we had. This work is supported by the PhD program AMX of \'Ecole Polytechnique and Ministère de l'Enseignement Supérieur et de la Recherche et de l’Innovation, and the European Research Council grant ALKAGE number 670846 managed by J.-P. Demailly.
  
\end{document}